\documentclass[twoside,leqno,10pt, A4]{amsart}
\usepackage{amsfonts}
\usepackage{amsmath}
\usepackage{amscd}
\usepackage{amssymb}
\usepackage{amsthm}
\usepackage{amsrefs}
\usepackage{latexsym}
\usepackage{mathrsfs}
\usepackage{bbm}
\usepackage{enumerate}
\usepackage{graphicx}
\usepackage{amscd}
\usepackage{color}
\begin{document}

\newtheorem{theorem}[subsection]{Theorem}
\newtheorem{proposition}[subsection]{Proposition}
\newtheorem{lemma}[subsection]{Lemma}
\newtheorem{corollary}[subsection]{Corollary}
\newtheorem{conjecture}[subsection]{Conjecture}
\newtheorem{prop}[subsection]{Proposition}
\numberwithin{equation}{section}
\newcommand{\mr}{\ensuremath{\mathbb R}}
\newcommand{\mc}{\ensuremath{\mathbb C}}
\newcommand{\dif}{\mathrm{d}}
\newcommand{\intz}{\mathbb{Z}}
\newcommand{\ratq}{\mathbb{Q}}
\newcommand{\natn}{\mathbb{N}}
\newcommand{\comc}{\mathbb{C}}
\newcommand{\rear}{\mathbb{R}}
\newcommand{\prip}{\mathbb{P}}
\newcommand{\uph}{\mathbb{H}}
\newcommand{\fief}{\mathbb{F}}
\newcommand{\majorarc}{\mathfrak{M}}
\newcommand{\minorarc}{\mathfrak{m}}
\newcommand{\sings}{\mathfrak{S}}
\newcommand{\fA}{\ensuremath{\mathfrak A}}
\newcommand{\mn}{\ensuremath{\mathbb N}}
\newcommand{\mq}{\ensuremath{\mathbb Q}}
\newcommand{\half}{\tfrac{1}{2}}
\newcommand{\f}{f\times \chi}
\newcommand{\summ}{\mathop{{\sum}^{\star}}}
\newcommand{\chiq}{\chi \bmod q}
\newcommand{\chidb}{\chi \bmod db}
\newcommand{\chid}{\chi \bmod d}
\newcommand{\sym}{\text{sym}^2}
\newcommand{\hhalf}{\tfrac{1}{2}}
\newcommand{\sumstar}{\sideset{}{^*}\sum}
\newcommand{\sumprime}{\sideset{}{'}\sum}
\newcommand{\sumprimeprime}{\sideset{}{''}\sum}
\newcommand{\sumflat}{\sideset{}{^\flat}\sum}
\newcommand{\shortmod}{\ensuremath{\negthickspace \negthickspace \negthickspace \pmod}}
\newcommand{\V}{V\left(\frac{nm}{q^2}\right)}
\newcommand{\sumi}{\mathop{{\sum}^{\dagger}}}
\newcommand{\mz}{\ensuremath{\mathbb Z}}
\newcommand{\leg}[2]{\left(\frac{#1}{#2}\right)}
\newcommand{\muK}{\mu_{\omega}}
\newcommand{\thalf}{\tfrac12}
\newcommand{\lp}{\left(}
\newcommand{\rp}{\right)}
\newcommand{\Lam}{\Lambda_{[i]}}
\newcommand{\lam}{\lambda}
\def\L{\fracwithdelims}
\def\om{\omega}
\def\pbar{\overline{\psi}}
\def\phis{\varphi^*}
\def\lam{\lambda}
\def\lbar{\overline{\lambda}}
\newcommand\Sum{\Cal S}
\def\Lam{\Lambda}
\newcommand{\sumtt}{\underset{(d,2)=1}{{\sum}^*}}
\newcommand{\sumt}{\underset{(d,2)=1}{\sum \nolimits^{*}} \widetilde w\left( \frac dX \right) }

\newcommand{\hf}{\tfrac{1}{2}}
\newcommand{\af}{\mathfrak{a}}
\newcommand{\Wf}{\mathcal{W}}

\newcommand{\MT}{\operatorname{MT}}
\newcommand{\ET}{\operatorname{ET}}
\newcommand{\PP}{\operatorname{Prob}}
\newcommand{\sm}{\operatorname{small}}
\newcommand{\lr}{\operatorname{large}}
\newcommand{\LCM}{\operatorname{LCM}}
\newcommand{\tp}{\operatorname{top}}
\newcommand{\spn}{\operatorname{span}}
\newcommand{\Vol}{\operatorname{Vol}}
\newcommand{\new}{\operatorname{new}}
\newcommand{\old}{\operatorname{old}}
\newcommand{\GCD}{\operatorname{GCD}}

\newcommand{\bC}{\mathbb{C}}
\newcommand{\bH}{\mathbb{H}}
\newcommand{\bN}{\mathbb{N}}
\newcommand{\bQ}{\mathbb{Q}}
\newcommand{\bR}{\mathbb{R}}
\newcommand{\bS}{\mathbb{S}}
\newcommand{\bT}{\mathbb{T}}
\newcommand{\bU}{\mathbb{U}}
\newcommand{\bZ}{\mathbb{Z}}
\newcommand{\zed}{\mathbb{Z}}

\newcommand{\Fix}{\mathrm{Fix}}
\newcommand{\sgn}{\mathrm{sgn}}
\newcommand{\Res}{\mathrm{Res}}
\newcommand{\rNW}{\mathrm{NW}}
\newcommand{\rKL}{\mathrm{Kl}}
\newcommand{\PSL}{\mathrm{PSL}}
\newcommand{\GL}{\mathrm{GL}}

\newcommand{\E}{\mathbf{E}}
\newcommand{\Var}{\mathbf{Var}}
\newcommand{\one}{\mathbf{1}}

\newcommand{\vlambda}{{\mathbf{\lambda}}}
\newcommand{\vrho}{{\mathbf{\rho}}}
\newcommand{\vdelta}{{\mathbf{\delta}}}
\newcommand{\vmu}{{\mathbf{\mu}}}
\newcommand{\vx}{{\mathbf{x}}}
\newcommand{\vr}{{\mathbf{r}}}
\newcommand{\vz}{{\mathbf{z}}}
\newcommand{\ve}{{\mathbf{e}}}
\newcommand{\vk}{{\mathbf{k}}}

\newcommand{\cC}{{\mathcal{C}}}
\newcommand{\cD}{{\mathcal{D}}}
\newcommand{\cB}{{\mathcal{B}}}
\newcommand{\cF}{{\mathcal{F}}}
\newcommand{\cN}{{\mathcal{N}}}
\newcommand{\cP}{{\mathcal{P}}}
\newcommand{\cI}{{\mathcal{I}}}
\newcommand{\cS}{{\mathcal{S}}}

\newcommand{\fa}{{\mathfrak{a}}}

\newcommand{\sB}{{\mathscr{B}}}
\newcommand{\sC}{{\mathscr{C}}}
\newcommand{\sD}{{\mathscr{D}}}
\newcommand{\sF}{{\mathscr{F}}}
\newcommand{\sG}{{\mathscr{G}}}
\newcommand{\sH}{{\mathscr{H}}}
\newcommand{\sI}{{\mathscr{I}}}
\newcommand{\sK}{{\mathscr{K}}}
\newcommand{\sL}{{\mathscr{L}}}
\newcommand{\sM}{{\mathscr{M}}}
\newcommand{\sN}{{\mathscr{N}}}
\newcommand{\sO}{{\mathscr{O}}}
\newcommand{\sP}{{\mathscr{P}}}
\newcommand{\sS}{{\mathscr{S}}}

\newcommand{\ug}{\underline{g}}
\newcommand{\ul}{\underline{\ell}}
\newcommand{\un}{\underline{n}}
\newcommand{\uw}{\underline{w}}
\newcommand{\ux}{\underline{x}}

\newcommand{\Ts}{\tilde{s}}
\newcommand{\Tt}{\tilde{t}}

\theoremstyle{plain}
\newtheorem{conj}{Conjecture}
\newtheorem{remark}[subsection]{Remark}

\makeatletter
\def\widebreve{\mathpalette\wide@breve}
\def\wide@breve#1#2{\sbox\z@{$#1#2$}%
     \mathop{\vbox{\m@th\ialign{##\crcr
\kern0.08em\brevefill#1{0.8\wd\z@}\crcr\noalign{\nointerlineskip}%
                    $\hss#1#2\hss$\crcr}}}\limits}
\def\brevefill#1#2{$\m@th\sbox\tw@{$#1($}%
  \hss\resizebox{#2}{\wd\tw@}{\rotatebox[origin=c]{90}{\upshape(}}\hss$}
\makeatletter

\title[Upper bounds for moments of Dirichlet $L$-functions to a fixed modulus]{Upper bounds for moments of Dirichlet $L$-functions to a fixed modulus}

\author[P. Gao]{Peng Gao}
\address{School of Mathematical Sciences, Beihang University, Beijing 100191, China}
\email{penggao@buaa.edu.cn}

\author[L. Zhao]{Liangyi Zhao}
\address{School of Mathematics and Statistics, University of New South Wales, Sydney NSW 2052, Australia}
\email{l.zhao@unsw.edu.au}

\begin{abstract}
We study the $2k$-th moment of central values of the family of Dirichlet $L$-functions to a fixed prime modulus and establish sharp upper bounds for all real $k \in [0,2]$.
\end{abstract}

\maketitle



\section{Introduction}
\label{sec 1}

  Let $L(s, \chi)$ be the $L$-function attached to a primitive Dirichlet character $\chi$ modulo $q$.  Throughout the paper, we assume $q \not \equiv 2 \pmod 4$ to ensure the existence of such primitive characters. We also write $\phis(q)$ for the number of primitive characters modulo $q$ and $\sum^{\star}$ the sum over primitive Dirichlet characters modulo $q$. It was conjectured by J. B. Conrey, D. W. Farmer, J. P. Keating, M. O. Rubinstein and N. C. Snaith in \cite{CFKRS} that for all positive integral values of $k$ and explicit constants $C_k$,
\begin{align}
\label{moments}
 \sumstar_{\substack{ \chi \shortmod q }}|L(\tfrac{1}{2},\chi)|^{2k} \sim C_k \phis(q)(\log q)^{k^2}.
\end{align}
In fact,  the above is believed (see \cite{R&Sound}) to hold for all real $k  \geq 0$. The case $k=1$ of \eqref{moments} follows from a result of A. Selberg \cite{Selberg46} on a more general twisted second moment.  The formula with $k=2$ of \eqref{moments} was first obtained by D. R. Heath-Brown \cite{HB81} for all most all $q$,  and was later improved by K. Soundararajan \cite{Sound2007} to hold for all $q$. An asymptotic formula for $k=2$ with a power saving error term was first achieved for primes $q$ by M. P. Young \cite{Young2011}. See \cites{BFKMM1,BFKMM,BM15,BFKMMS} for subsequent improvements on the error term in Young's result. The result in \cite{Young2011} was extended to be valid  for all sufficiently factorable $q$ including $99.9\%$ of all admissible moduli by V. Blomer and D. Mili\'cevi\'c \cite{BM15}, and was further shown to hold for all general moduli $q$ by X. Wu \cite{Wu2020}. \newline
 
 Besides the asymptotic formula \eqref{moments}, we have upper and lower bounds of the conjectured order of magnitude for the expression on the left-hand side of \eqref{moments}.  See \cites{BPRZ, Sound01,HB2010, Harper, R&Sound1, Radziwill&Sound1, C&L, Gao2024, AC25} for the related results.  Note that in \cite{Gao2024}, using the upper bounds principle of M. Radziwi{\l\l} and K. Soundararajan  \cite{Radziwill&Sound} and the lower bounds principle of W. Heap and K. Soundararajan  \cite{H&Sound}, it is shown that for large prime $q$, 
\begin{align}
\label{lowerupperbound}
\begin{split}
   \sumstar_{\substack{ \chi \shortmod q }}|L(\tfrac{1}{2},\chi)|^{2k} \gg_k & \phis(q)(\log q)^{k^2}, \quad \text{for all } k \geq 0, \\
    \sumstar_{\substack{ \chi \shortmod q }}|L(\tfrac{1}{2},\chi)|^{2k} \ll_k & \phis(q)(\log q)^{k^2}, \quad \text{for } 0 \leq k \leq 1.
\end{split}
\end{align}  
  
  It is remarked in \cite{BPRZ} that one may apply the work of B. Hough \cite{Hough2016} or R. Zacharias
   \cite{Z2019} on twisted fourth moment for the family of Dirichlet $L$-functions under consideration to obtain sharp upper bounds on all moments
   below the fourth. It is the aim of this paper to achieve this. Our main result is as follows. 
\begin{theorem}
\label{thmupperbound}
   For large prime $q$ and any real number $k$ with $0 \leq k \leq 2$, we have
\begin{align}
\label{upperbound}
   \sumstar_{\substack{ \chi \shortmod q }}|L(\tfrac{1}{2},\chi)|^{2k} \ll_k \phis(q)(\log q)^{k^2}.
\end{align}
\end{theorem}

Combining \eqref{lowerupperbound} and \eqref{upperbound}, we immediately deduce the next result on the order of magnitude of the family of $L$-functions under our consideration.
\begin{theorem}
\label{thmorderofmag}
   For large prime $q$ and any real number $k$ such that $0 \leq k \leq 2$, we have
\begin{align*}
   \sumstar_{\substack{ \chi \shortmod q }}|L(\tfrac{1}{2},\chi)|^{2k} \asymp \phis(q)(\log q)^{k^2}.
\end{align*}
\end{theorem}

   Our proof of Theorem \ref{thmupperbound} is based on the above mentioned upper bounds principle of M. Radziwi{\l\l} and K. Soundararajan \cite{Radziwill&Sound} and follows the approach in \cite{Gao2024}. The proof also relies crucially on the twisted fourth moment of the relevant family of $L$-functions instead of merely the fourth moment, since the arguments require the evaluation the mollified fourth moment (see Proposition \ref{Prop5} below). This roadmap of deriving sharp upper bound for all moments below the fourth of the family Dirichlet $L$-functions modulo $q$ was mentioned in \cite{BPRZ}.  We shall discuss this in detail in the next section.  It will also be important that the relevant Dirichlet polynomials, e.g. \eqref{defPQ}, \eqref{defN}, are short so that certain off-diagonal contribution can be controlled. 

\section{Twisted fourth moment of Dirichlet $L$-functions}
  
  For the purpose of this paper, we need to compute a more general fourth moment, namely, that of central values of Dirichlet $L$-functions to a fixed modulus $q$ twisted by Dirichlet characters evaluated at integers $\ell_1$, $\ell_2$.  Such a result was first obtained by B. Hough \cite[Theorem 4]{Hough2016} (the proof of the result can be found in the preprint arXiv:1304.1241v3) for square-free $\ell_1, \ell_2$. The result was extended to hold for cubic-free $\ell_1$, $\ell_2$ by R. Zacharias \cite{Z2019},  and for general $\ell_1$, $\ell_2$ by D. Liu \cite[Theorem 6.2]{Liu24}.  We actually require the main term to be more explicit than that stated in \cite[Theorem 6.2]{Liu24} which involves with sums of smoothed Dirichlet series.  In fact, such a result has already been derived in the section entitled ``Deduction of the main terms" in \cite{Liu24}. Upon writing the smoothed weight using its integral representation and shifting the contour of integration to the right of $-1/4$ line to pick up the poles arising from $0$ as did in \cite{Hough2016}, we recover readily a more concrete main term matching the one in \cite[Theorem 4]{Hough2016} for general $\ell_1$, $\ell_2$, with an admissible error term. In this process, we also make use of an generalization of Lemma 31 of the above-mentioned arXiv preprint of Hough. \newline
  
To be more precise, we define for any positive integer $n$ and any real number $\lambda$, 
\begin{align*}
\sigma_\lambda(n)=\sum_{d\mid n}d^\lambda. 
\end{align*}  
   We then define for positive integers $\ell, m$ and real numbers $\lambda, \nu, v$
\begin{align*}
&\tau_{\lambda, \nu, v}(\ell;m) = \prod_{p|\ell} \left(\sum_{j \geq 0} \sigma_\lambda(p^{j + m}) \sigma_\nu(p^j)p^{-jv}\right)\bigg/ \prod_{p|\ell}\left(\sum_{j\geq 0} \sigma_\lambda(p^j)\sigma_\nu(p^j)p^{-jv}\right),
\end{align*}
 where we reserve the letter $p$ for a prime number throughout the paper. \newline
 
Then a straightforward generalization of Lemma 31 of Hough's preprint (corresponding to $m=1$ here) gives that 
\begin{align*}
&\tau_{\lambda, \nu, v}(p;m) = 1 + \frac{p^\lambda(p^{\lambda m}-1) \zeta_p(2v-\lambda - \nu)}{(p^{\lambda}-1)\zeta_p(v)\zeta_p(v-\nu)},
\end{align*}
  where $\zeta_p(s) = (1-p^{-s})^{-1}$. We shall also write $\zeta(s)$ for the Riemann zeta function. \newline
  
   We further define for any positive integer $\ell$ and real numbers $\alpha, \beta, \gamma, \delta \in \bC$, the generalized divisor function
\begin{align}
\label{taudef}
 \tau_{\alpha, \beta, \gamma, \delta}(\ell) = \prod_{\substack{ p| \ell \\ p^{\nu}\|\ell}}\left(1 + \frac{p^{\gamma-\delta}(p^{(\gamma-\delta)\nu}-1) \zeta_p(2 + \alpha + \beta + \gamma + \delta)}{(p^{(\gamma-\delta)}-1)\zeta_p(1 +\alpha + \gamma)\zeta_p(1 + \beta + \gamma)}\right).
\end{align}
  
  It follows, from \cite[Theorem 6.2]{Liu24} and our discussions above, that \cite[Theorem 4]{Hough2016} is still valid for general $\ell_1$, $\ell_2$, except that we replace the generalized divisor functions in \cite[Theorem 4]{Hough2016} by the one defined above. Now, we are ready to state the necessary twisted fourth moment result.
\begin{proposition}
\label{twisted_fourth_moment_theorem} With the notation as above, let $0 \leq \vartheta < 1/32$ and $1 \leq \ell_1, \ell_2 \leq q^{\vartheta}$ be positive integers such that $(\ell_1, \ell_2) = 1$. For a set of indeces $S$, let
\[
X_S = \prod_{u \in S} X_u, \quad \mbox{where} \quad X_u = \left(\frac{q}{\pi}\right)^{-u}\frac{\Gamma\left(\frac{\frac{1}{2}-u}{2}\right)}{\Gamma\left(\frac{\frac{1}{2}+u}{2}\right)}. 
\]
 Then there exists $\eta > 0$ such that for $\alpha, \beta, \gamma, \delta \in \left\{z \in \bC: |z| < \eta/\log q \right\}$, for any $\varepsilon>0$, 
\begin{equation} \label{4thmomenteval}
\begin{split}
 M(\alpha, & \beta, \gamma, \delta; \ell_1,\ell_2) := 
 \sumstar_{\substack{ \chi \shortmod q \\ \chi(-1)=1}} L\left(\tfrac{1}{2} + \alpha, \chi \right)L\left(\tfrac{1}{2} + \beta, \chi \right)L\left(\tfrac{1}{2} + \gamma, \overline{\chi} \right)L\left(\tfrac{1}{2} + \delta, \overline{\chi} \right)  \chi(\ell_1)\overline{\chi}(\ell_2)
\\
=& \frac{\tau_{ \alpha, \beta, \gamma, \delta}(\ell_1)\tau_{ \gamma, \delta, \alpha, \beta}(\ell_2)}{\ell_1^{1/2 + \gamma} \ell_2^{1/2+\alpha}}\frac{\zeta(1 +\alpha + \gamma)\zeta(1 + \alpha + \delta) \zeta(1 + \beta + \gamma) \zeta(1 + \beta + \delta)}{\zeta(2 +\alpha + \beta +\gamma+\delta)}\\
& \hspace*{1cm}+ X_{\alpha,\gamma}\frac{\tau_{-\gamma, \beta, -\alpha, \delta}(\ell_1) \tau_{-\alpha, \delta, -\gamma, \beta}(\ell_2)}{\ell_1^{1/2 -\alpha} \ell_2^{1/2 -\gamma}} 
 \frac{\zeta(1-\alpha -\gamma) \zeta(1-\gamma +\delta) \zeta(1 -\alpha + \beta) \zeta(1+\beta + \delta)}{\zeta(2-\alpha + \beta -\gamma +\delta)}\\ 
 & \hspace*{1cm}+ X_{\beta,\gamma}\frac{\tau_{\alpha, -\gamma, -\beta, \delta}(\ell_1) \tau_{-\beta, \delta, \alpha, -\gamma}(\ell_2)}{\ell_1^{1/2 -\beta} \ell_2^{1/2 + \alpha}} 
 \frac{\zeta(1+\alpha - \beta) \zeta(1+\alpha +\delta) \zeta(1-\beta-\gamma) \zeta(1-\gamma + \delta)}{\zeta(2+\alpha - \beta -\gamma +\delta)}\\
 & \hspace*{1cm}+ X_{\alpha,\delta}\frac{\tau_{-\delta, \beta, \gamma, -\alpha}(\ell_1) \tau_{\gamma, -\alpha, -\delta, \beta}(\ell_2)}{\ell_1^{1/2 +\gamma} \ell_2^{1/2 - \delta}}  
  \frac{\zeta(1+\gamma-\delta) \zeta(1-\alpha -\delta) \zeta(1 + \beta + \gamma) \zeta(1-\alpha+\beta)}{\zeta(2-\alpha + \beta +\gamma -\delta)}\\
 &  \hspace*{1cm}+ X_{\beta,\delta}\frac{\tau_{\alpha, -\delta, \gamma, -\beta}(\ell_1) \tau_{\gamma, -\beta, \alpha, -\delta}(\ell_2)}{\ell_1^{1/2 +\gamma} \ell_2^{1/2 + \alpha}} 
   \frac{\zeta(1+\alpha +\gamma) \zeta(1+\alpha-\beta) \zeta(1 +\gamma - \delta) \zeta(1-\beta - \delta)}{\zeta(2+\alpha - \beta +\gamma -\delta)}\\
  &  \hspace*{1cm}+  X_{\alpha, \beta, \gamma, \delta}\frac{\tau_{ -\gamma, -\delta, -\alpha, -\beta}(\ell_1)\tau_{ -\alpha, -\beta, -\gamma, -\delta }(\ell_2)}{\ell_1^{1/2 -\alpha} \ell_2^{1/2-\gamma}}
     \frac{\zeta(1 -\alpha - \gamma)\zeta(1 - \beta - \gamma) \zeta(1 - \alpha - \delta) \zeta(1 - \beta - \delta)}{\zeta(2 -\alpha - \beta -\gamma-\delta)}\\
  & \hspace*{1cm}+ O\left(\frac{\max(\ell_1,\ell_2)^{1/2}}{q^{1/32-\varepsilon} \min(\ell_1,\ell_2)^{1/2}}\right). 
\end{split}
\end{equation}
\end{proposition}
 
  We note that the main term in the above expression is consistent with a conjectural formula of J. B. Conrey, D. W. Farmer, J. P. Keating,
  M. O. Rubinstein and N. C. Snaith in \cite{CFKRS}.  Note further that $M(\alpha, \beta, \gamma, \delta; \ell_1,\ell_2)$ must be invariant under the swapping of variables, $\alpha \leftrightarrow \beta$, which implies in particular that (as one can check directly using the definition of $\tau$ in \eqref{taudef}) for any real numbers $\alpha$, $\beta$, $\gamma$, $\delta$ and any positive integer $\ell$, 
\begin{align*}
\begin{split}
  \tau_{\alpha, \beta, \gamma, \delta}(\ell) = \tau_{\beta, \alpha,  \gamma, \delta}(\ell), \quad
  \frac{\tau_{\alpha, \beta, \gamma, \delta}(\ell)}{\ell^{1/2+\gamma}} =& \frac{\tau_{\alpha, \beta, \delta, \gamma}(\ell)}{\ell^{1/2 + \delta}}.
\end{split} 
\end{align*}    
  
We are in fact interested in the limit when all parameters $\alpha$, $\beta$, $\gamma$, $\delta$ are taken to zero.  We note first that, as already pointed out in \cite{Young2011}, both the main and the error terms in \eqref{4thmomenteval} are holomorphic with respect to the shift parameters $\alpha$, $\beta$, $\gamma$, $\delta$. Hence in the limit, we may set aside the error term (as it remains the same) as well as the contributions from the poles (as they cancel each other) in the main term and concentrate only on the computations on the constant term arising from the main term. To do so, we first set $\alpha=0$ in the main term above and it becomes
\begin{align*}
& \frac{\tau_{ 0, \beta, \gamma, \delta}(\ell_1)\tau_{ \gamma, \delta, 0, \beta}(\ell_2)}{\ell_1^{1/2 + \gamma} \ell_2^{1/2}}\frac{\zeta(1 + \gamma)\zeta(1  + \delta) \zeta(1 + \beta + \gamma) \zeta(1 + \beta + \delta)}{\zeta(2  + \beta +\gamma+\delta)}\\
& \hspace*{1cm}+ X_{\gamma}\frac{\tau_{-\gamma, \beta, 0, \delta}(\ell_1) \tau_{0, \delta, -\gamma, \beta}(\ell_2)}{\ell_1^{1/2} \ell_2^{1/2 -\gamma}} 
 \frac{\zeta(1-\gamma) \zeta(1-\gamma +\delta) \zeta(1  + \beta) \zeta(1+\beta + \delta)}{\zeta(2 + \beta -\gamma +\delta)}\\ 
& \hspace*{1cm}+ X_{\beta,\gamma}\frac{\tau_{0, -\gamma, -\beta, \delta}(\ell_1) \tau_{-\beta, \delta, 0, -\gamma}(\ell_2)}{\ell_1^{1/2 -\beta} \ell_2^{1/2}} 
 \frac{\zeta(1 - \beta) \zeta(1 +\delta) \zeta(1-\beta-\gamma) \zeta(1-\gamma + \delta)}{\zeta(2 - \beta -\gamma +\delta)}\\
 & \hspace*{1cm}+ X_{\delta}\frac{\tau_{-\delta, \beta, \gamma, 0}(\ell_1) \tau_{\gamma, 0, -\delta, \beta}(\ell_2)}{\ell_1^{1/2 +\gamma} \ell_2^{1/2 - \delta}}  
  \frac{\zeta(1+\gamma-\delta) \zeta(1 -\delta) \zeta(1 + \beta + \gamma) \zeta(1+\beta)}{\zeta(2 + \beta +\gamma -\delta)}\\
  & \hspace*{1cm}+ X_{\beta,\delta}\frac{\tau_{0, -\delta, \gamma, -\beta}(\ell_1) \tau_{\gamma, -\beta, 0, -\delta}(\ell_2)}{\ell_1^{1/2 +\gamma} \ell_2^{1/2}} 
   \frac{\zeta(1 +\gamma) \zeta(1-\beta) \zeta(1 +\gamma - \delta) \zeta(1-\beta - \delta)}{\zeta(2 - \beta +\gamma -\delta)}\\
  &  \hspace*{1cm} +  X_{\beta, \gamma, \delta}\frac{\tau_{ -\gamma, -\delta, 0, -\beta}(\ell_1)\tau_{ 0, -\beta, -\gamma, -\delta }(\ell_2)}{\ell_1^{1/2} \ell_2^{1/2-\gamma}}
     \frac{\zeta(1 - \gamma)\zeta(1 - \beta - \gamma) \zeta(1 - \delta) \zeta(1 - \beta - \delta)}{\zeta(2  - \beta -\gamma-\delta)}\\
  & =:  S_1+S_2+S_3+S_4+S_5+S_6, \quad \mbox{say}.
\end{align*}   

We then take $\gamma \rightarrow 0$ followed by $\delta \rightarrow -\beta$. By doing so, the sum $S_3+S_4$ produces no poles in the process, so that
\begin{align}
\label{limnopole}
\begin{split}
 & \lim_{\delta \rightarrow -\beta} \Big (\lim_{\gamma \rightarrow 0} (S_3+S_4) \Big ) \\
=& \lim_{\delta \rightarrow -\beta} \Big (X_{\beta}\frac{\tau_{0, 0, -\beta, \delta}(\ell_1) \tau_{-\beta, \delta, 0, 0}(\ell_2)}{\ell_1^{1/2 -\beta} \ell_2^{1/2}} 
 \frac{\zeta(1 - \beta)^2 \zeta(1 +\delta)^2 }{\zeta(2 - \beta  +\delta)}+ X_{\delta}\frac{\tau_{-\delta, \beta, 0, 0}(\ell_1) \tau_{0, 0, -\delta, \beta}(\ell_2)}{\ell_1^{1/2} \ell_2^{1/2 - \delta}}  
  \frac{\zeta(1-\delta)^2   \zeta(1 + \beta )^2}{\zeta(2 + \beta  -\delta)} \Big ) \\
=& X_{\beta}\frac{\tau_{0, 0, -\beta, -\beta}(\ell_1) \tau_{-\beta, -\beta, 0, 0}(\ell_2)}{\ell_1^{1/2 -\beta} \ell_2^{1/2}} 
 \frac{\zeta(1 - \beta)^4 }{\zeta(2 - 2\beta)} + X_{-\beta}\frac{\tau_{\beta, \beta, 0, 0}(\ell_1) \tau_{0, 0, \beta, \beta}(\ell_2)}{\ell_1^{1/2} \ell_2^{1/2 + \beta}}  
  \frac{  \zeta(1 + \beta )^4}{\zeta(2 +2 \beta  )}.  
 \end{split}
\end{align}      
    
     Note moreover that by \cite[Corollary 1.16]{MVa1}, the Laurent expansion of $\zeta(s)$ $s=1$ gives 
\begin{align}
\label{zetaexpansion}
\begin{split}
 \zeta(s)=\frac 1{s-1}+\gamma_0+O(s-1), \quad |s-1| \leq 1,
\end{split}
\end{align}        
   where $\gamma_0$ is Euler's constant. \newline
   
   The above implies that both the terms $S_1$ and $S_2$ have a simple pole under the limit $\gamma \rightarrow 0$.  However, the poles of $\zeta(1+\gamma)$ and $\zeta(1-\gamma)$ cancel each other. We thus apply \eqref{zetaexpansion} to see that, 
\begin{align}
\label{limgamma0S1S2}
\begin{split}
 \lim_{\gamma \rightarrow 0} (S_1+S_2) =& \frac{\tau_{ 0, \beta, 0, \delta}(\ell_1)\tau_{ 0, \delta, 0, \beta}(\ell_2)}{\ell_1^{\frac{1}{2}} \ell_2^{\frac{1}{2}}}\frac{\zeta(1  + \delta) \zeta(1 + \beta ) \zeta(1 + \beta + \delta)}{\zeta(2  + \beta +\delta)}\\
& \hspace*{1cm} \times \Big (\log \frac{q}{\pi}+C_0+ \frac {\zeta'(1+\beta )}{\zeta(1+\beta)}+\frac {\zeta'(1 +\delta)}{\zeta(1 +\delta)}-2\frac{\zeta'(2+\beta +\delta)}{\zeta(2+\beta +\delta)} \\
& \hspace*{1.5cm} +\sum_{\substack{ p| \ell_1 \\ p^\nu\|\ell_1}}\left(1 + \frac{p^{\delta}(p^{\delta \nu}-1) \zeta_p(2 +\beta + \delta)}{(p^{\delta}-1)\zeta_p(1+\delta)\zeta_p(1 + \beta +\delta )}\right)^{-1}\\
 & \hspace*{2cm} \times \Big (\frac{p^{\delta}(p^{\delta \nu}-1) (2\zeta'_p(2 +\beta  + \delta)-\zeta'_p(1  + \delta))}{(p^{\delta}-1)\zeta_p(1 +\delta)\zeta_p(1 + \beta + \delta)} \\
 &\hspace*{2.5cm} -\frac{\log p^\nu p^{\delta \nu}((p^{\delta}-1)^2+(p^{\delta}-1))-\log p p^{\delta}(p^{\delta \nu}-1)) \zeta_p(2 +\beta  + \delta)}{(p^{\delta}-1)^2\zeta_p(1 +\delta)\zeta_p(1 + \beta + \delta)} \Big )  \\
& \hspace*{1.5cm} +\sum_{\substack{ p| \ell_2 \\ p^\nu\|\ell_2}}\left(1 + \frac{p^{\beta}(p^{\beta \nu}-1) \zeta_p(2 +\beta + \delta)}{(p^{\beta}-1)\zeta_p(1 +\beta)\zeta_p(1+\beta+\delta)}\right)^{-1} \\
 & \hspace*{2cm} \times \Big (\frac{p^{\beta}(p^{\beta \nu}-1) (2\zeta'_p(2 +\beta  + \delta)-\zeta'_p(1  + \beta))}{(p^{\beta}-1)\zeta_p(1 +\beta)\zeta_p(1 + \beta + \delta)} \\
 & \hspace*{2.5cm} -\frac{\log p^\nu p^{\beta \nu}((p^{\beta}-1)^2+(p^{\beta}-1))-\log p p^{\beta}(p^{\beta \nu}-1)) \zeta_p(2 +\beta  + \delta)}{(p^{\beta}-1)^2\zeta_p(1 +\beta)\zeta_p(1 + \beta + \delta)} \Big ) \Big), 
\end{split}
\end{align}          
    where $C_0$ is a constant involving $\gamma_0$ and derivatives of $\Gamma(s)$.  Now $\lim_{\gamma \rightarrow 0} (S_5+S_6)$ has a similar expression by observing that other than a factor $X_{\beta,\delta}$, the sum $S_5+S_6$ can be obtained from $S_1+S_2$ by a change of variables: $\beta \rightarrow -\delta$, $\delta \rightarrow -\beta$. \newline
    
Next note that the right-hand side of \eqref{limgamma0S1S2} has a simple pole under the limit $\delta \rightarrow -\beta$. We then apply \eqref{zetaexpansion} to compute $\lim_{\delta \rightarrow -\beta}(\lim_{\gamma \rightarrow 0} (S_1+S_2))$ and similarly $\lim_{\delta \rightarrow -\beta}(\lim_{\gamma \rightarrow 0} (S_5+S_6))$. The resulting expressions are now functions of $\beta$ only. We now use them and the expression given in \eqref{limnopole} to compute the limit $\beta \rightarrow 0$. Here we remark again that the expressions remain holomorphic in the process. For example, the fourth order pole on the right-hand side of \eqref{limnopole} cancels with those coming from $\lim_{\delta \rightarrow -\beta}(\lim_{\gamma \rightarrow 0} (S_1+S_2))$ (which involve the terms $\zeta'(1-\beta)\zeta'(1+\beta)$) and from $\zeta''(1-\beta)\zeta(1+\beta)$) and $\lim_{\delta \rightarrow -\beta}(\lim_{\gamma \rightarrow 0} (S_5+S_6))$ (which involve the terms $\zeta''(1+\beta)\zeta(1-\beta)$) and $\zeta'(1-\beta)\zeta'(1+\beta)$) while noting that $\zeta''(s) \sim \frac 2{(s-1)^3}$ when $s$ in near $1$. Similar observations apply to other poles. \newline

For a function $f(x)=\prod^k_{j=1}h(x)$, the product rule (and induction) implies that
\begin{align*}
  f^{(n)}(x)=n!\sum_{\substack{\alpha_1+\ldots+\alpha_k =n \\ \alpha_1 \geq 0, \ldots, \alpha_k \geq 0}}\frac {h^{(\alpha_1)}(x)}{\alpha_1!}\cdot \frac {h^{(\alpha_2)}(x)}{\alpha_2!}\cdot\ldots \cdot \frac {h^{(\alpha_k)}(x)}{\alpha_k!}. 
\end{align*}     
   
    We apply the above to compute the constant term in the expression $\lim_{\delta \rightarrow -\beta}(\lim_{\gamma \rightarrow 0} \sum^6_{j=1}S_j)$, regarded as 
function of $\beta$.  In particular, this applies to the evaluation on the $j$-th order ($0 \leq j \leq 4$) derivative of various $\tau_{ *, *, *, *}(\ell)$ for $\ell=\ell_1, \ell_2$ evaluated at $\beta=0$, where $*$ takes value in the set $\{0, \beta, -\beta \}$. Inspection of these derivatives allow us to write them in the form
\begin{align*}
  \sum_{\substack{0 \leq j_1+\ldots+j_4 \leq 4 \\ j_1 \geq 0, \ldots, j_4 \geq 0}}\sum_{\substack{ 1 \leq i \leq 4 \\ p^{\nu_i}_i \| \ell, \nu_i \geq 1 \\ p_i \text{distinct} }}  C_{j_1, \ldots, j_4,i}(p_i)(\log p^{\nu_i}_i)^{j_i},
\end{align*}   
 where $C_{j_1, \ldots, j_4,i}(p_i)$ are constants depending on $p_i$ only such that $C_{j_1, \ldots, j_4,i}(p_i) \ll 1$. \newline
  
  Taking into account of the derivatives of other factor, as well as the observation that for the $\tau_{ *, *, *, *}(\ell)$ considered above, 
\begin{align*}
  \lim_{\beta \rightarrow 0}\tau_{ *, *, *, *}(\ell)=\prod_{\substack{ p| \ell \\ p^{\nu}\|\ell}}\left(1 + \frac{\nu (1-1/p)}{1+1/p}\right) =: \tau_0(\ell),
\end{align*}   
  we deduce now the following result on the twisted fourth moment of Dirichlet $L$-functions at the central point.
\begin{proposition}
\label{twisted_fourth_moment_theorem on the central point} With the notation as above, let $0 \leq \vartheta < 1/32$ and $1 \leq \ell_1, \ell_2 \leq q^{\vartheta}$ be positive integers with $(\ell_1, \ell_2) = 1$. Then for any $\varepsilon>0$, 
\begin{align}
\label{twistedfourthmoment}
\begin{split}
  \sumstar_{\substack{ \chi \shortmod q \\ \chi(-1)=1}}  \Big|L\left(\tfrac{1}{2}, \chi \right)\Big |^4\chi(\ell_1)\overline{\chi}(\ell_2) = \frac{\phis(q) \tau_{0}(\ell_1)\tau_{ 0}(\ell_2)}{\ell_1^{1/2} \ell_2^{1/2} \zeta(2 )} & \Big (\sum_{\substack{0 \leq j_0+j_1+\ldots+j_4 \leq 4 \\ j_0 \geq 0, j_1 \geq 0, \ldots, j_4 \geq 0}}(\log q)^{j_0}\sum_{\substack{ 1 \leq i \leq 4 \\ p^{\nu_i}_i \| \ell_1\ell_2, \nu_i \geq 1 \\ p_i \text{distinct} }}  C_{j_0, \ldots, j_4,i}(p_i)(\log p^{\nu_i}_i)^{j_i}\Big ) \\
  & + O\left(\frac{\max(\ell_1,\ell_2)^{1/2}}{q^{1/32-\varepsilon} \min(\ell_1,\ell_2)^{1/2}}\right),
\end{split}
\end{align}
  where $C_{j_0, \ldots, j_4,i}(p_i)$ are constants whose values depend on $p_i$ only and $C_{j_0, \ldots, j_4,i}(p_i) \ll 1$. Moreover, $C_{j_0, 0, 0, 0,0,i}(p_i)$ are absolute constants with expressions involving $\gamma_0$, $\Gamma'(\frac 14)$ and $\Gamma(\frac 14)$. 
\end{proposition}

\section{Proofs of Theorem \ref{thmupperbound}}
\label{sec 2}

\subsection{Initial Treatments}
\label{sec 2'}

Observe first that we have $\phis(q)=q-2$ for a prime $q$. As the case $0 \leq k \leq 1$ has been established in \cite[Theorem 1.2]{Gao2024} and the case $k=2$ is a consequence of Proposition \ref{twisted_fourth_moment_theorem on the central point} by setting $\ell_1=\ell_2$ in \eqref{twistedfourthmoment}, we shall assume that $1 < k < 2$ throughout the proof.  We define a sequence of even natural numbers $\{ \ell_j \}_{1 \leq j \leq R}$ such that $\ell_1= 2\lceil N \log \log q\rceil$ and $\ell_{j+1} = 2 \lceil N \log \ell_j \rceil$ for $j \geq 1$, where $N, M$ are two large natural numbers depending on $k$ only and $R$ denotes the largest natural number satisfying $\ell_R >10^M$. \newline
  
Write ${ P}_1$ the set of odd primes not exceeding $q^{1/\ell_1^2}$ and ${ P_j}$ the set of primes lying in the interval $(q^{1/\ell_{j-1}^2}, q^{1/\ell_j^2}]$ for $2\le j\le R$.  Further, for each $1 \leq j \leq R$, set
\begin{equation}
\label{defPQ}
{\mathcal P}_j(\chi) = \sum_{p\in P_j} \frac{\chi(p)}{\sqrt{p}}  \quad \mbox{and} \quad  {\mathcal Q}_j(\chi, k) =\Big (\frac{12 k^2 {\mathcal
P}_j(\chi) }{\ell_j}\Big)^{r_k\ell_j},
\end{equation}
where $r_k=\lceil k /(2k-1) \rceil+1$. We also define ${\mathcal Q}_{R+1}(\chi, k)=1$. \newline

  Moreover, for each $1 \leq j \leq R$ and any real number $\alpha$, we define
\begin{align}
\label{defN}
{\mathcal N}_j(\chi, \alpha) = E_{\ell_j} (\alpha {\mathcal P}_j(\chi)), \quad \mathcal{N}(\chi, \alpha) = \prod_{j=1}^{R} {\mathcal
N}_j(\chi,\alpha),
\end{align}
   where for any non-negative integer $\ell$ and any real number $x$,
\begin{equation*}
E_{\ell}(x) = \sum_{j=0}^{\ell} \frac{x^{j}}{j!}.
\end{equation*}
 We shall follow the convention throughout the paper that an empty product equals $1$. \newline

   We note that, as already pointed out in \cite{H&Sound}, the definition of ${\mathcal P}_j(\chi)$ is introduced so that the real part of the sum of them is
  used to approximate $\log |L(1/2, f \otimes \chi)|$, and the definition of ${\mathcal N}_j(\chi, \alpha)$ is introduced so that it approximates
  $\exp(\alpha{\mathcal P}_j(\chi))$. Thus $\mathcal{N}(\chi, \alpha)$ can be viewed as an approximation to $L(1/2, f \otimes \chi)^{\alpha}$.

  We now apply the upper bounds principle of M. Radziwi{\l\l} and K. Soundararajan in \cite{Radziwill&Sound} as in the proof of \cite[Lemma 3.2]{Gao2024} to prove the following analogue of \cite[Lemma 3.2]{Gao2024}.
\begin{lemma}
\label{lem2}
 With the notation as above, we have, for $1<k < 2$,
\begin{align}
\label{basiclowerbound2}
\begin{split}
\sumstar_{\substack{ \chi \shortmod q }}|L(\tfrac{1}{2},\chi)|^{2k} \ll \Big ( \sumstar_{\substack{ \chi \shortmod q }} & |L(\tfrac{1}{2},\chi)|^4 \sum^{R}_{v=0} \prod^v_{j=1}\Big ( |\mathcal{N}_j(\chi, k-2)|^2 \Big ) |{\mathcal
 Q}_{v+1}(\chi, k)|^{2}
 \Big)^{k/2} \\
 & \times \Big ( \sumstar_{\substack{ \chi \shortmod q }}  \sum^{R}_{v=0}  \Big (\prod^v_{j=1}|\mathcal{N}_j(\chi, k)|^2\Big )|{\mathcal
 Q}_{v+1}(\chi, k)|^{2} \Big)^{(2-k)/2},
\end{split}
\end{align}
  where the implied constant depends on $k$ alone.
\end{lemma}
\begin{proof}
Using arguments similar to those in the proof of \cite[Lemma 3.2]{Gao2024}, we get that when $|{\mathcal P}_j(\chi)| \le \ell_j/10$, uniformly for all $j$,
\begin{align}
\label{prodNlowerbound}
 |\mathcal{N}_j(\chi, k-2)|^{2k}|\mathcal{N}_j(\chi, k)|^{2(2-k)} \geq 1+ O\big(e^{-\ell_j} \big ).
\end{align}
  
Suppose that there is an integer $0 \leq v \leq R-1$ such that $| \mathcal{P}_j (\chi) | \leq \ell_j/10$ for all $j \leq v$ and that $|
  \mathcal{P}_{v+1} (\chi) | > \ell_{v+1}/10$.  Then in this case, $|{\mathcal Q}_{v+1}(\chi, k)| \geq 1$. Together with \eqref{prodNlowerbound}, we get that
\begin{align*}
 \Big ( \prod^v_{j=1}|\mathcal{N}_j(\chi, k-2)|^{2k}|\mathcal{N}_j(\chi, k)|^{2(2-k)} \Big )|{\mathcal Q}_{v+1}(\chi, k)|^4 \gg 1.
\end{align*}

On the other hand, if $| \mathcal{P}_j (\chi) | \leq \ell_j/10$ for all $1 \leq j \leq R$, then \eqref{prodNlowerbound} holds for all $j$ and hence
\begin{align*}
 \prod^R_{j=1}|\mathcal{N}_j(\chi, k-2)|^{2k}|\mathcal{N}_j(\chi, k)|^{2(2-k)} \gg 1.
\end{align*}

  In both cases, we derive that
\begin{align*}
 \Big (\sum^R_{v=0}\Big ( \prod^v_{j=1}|\mathcal{N}_j(\chi, k-2)|^2 \Big )|{\mathcal Q}_{v+1}(\chi, k)|^2 \Big )^{k/2}\Big (\sum^R_{v=0}\Big ( \prod^v_{j=1}|\mathcal{N}_j(\chi, k)|^2\Big ) |{\mathcal Q}_{v+1}(\chi, k)|^2 \Big )^{(2-k)/2}  \gg 1.
\end{align*}
The above implies that
\begin{align*}
 \sumstar_{\substack{ \chi \shortmod q }} |L(\half, \chi)|^{2k}  \ll \sumstar_{\substack{ \chi \shortmod q }} |L(\half, \chi)|^{2k} \Big (\sum^R_{v=0} & \Big ( \prod^v_{j=0}|\mathcal{N}_j(\chi, k-2)|^2 \Big )|{\mathcal Q}_{v+1}(\chi, k)|^2 \Big )^{k/2} \\
& \times \Big (\sum^R_{v=0}\Big (  \prod^v_{j=1}|\mathcal{N}_j(\chi, k)|^2\Big ) |{\mathcal Q}_{v+1}(\chi, k)|^2 \Big )^{(2-k)/2}.
\end{align*}
Now \eqref{basiclowerbound2} emerges from applying H\"older's inequality to the above.  This completes the proof of the lemma.
\end{proof}

 It follows from Lemma \ref{lem2} that in order to establish Theorem \ref{thmupperbound}, it suffices to prove the following two propositions.
\begin{proposition}
\label{Prop5} With the notation as above, we have
\begin{align*}
\sumstar_{\substack{ \chi \shortmod q }}|L(\tfrac{1}{2},\chi)|^4 \sum^{R}_{v=0}\Big (\prod^v_{j=1}|\mathcal{N}_j(\chi, k-2)|^{2}\Big ) |{\mathcal
 Q}_{v+1}(\chi, k)|^2    \ll
\phis(q)(\log q)^{ k^2 }.
\end{align*}
\end{proposition}

\begin{proposition}
\label{Prop6} With the notation as above, we have
\begin{align*}
 \sumstar_{\substack{ \chi \shortmod q }} \sum^{R}_{v=0} \Big ( \prod^v_{j=1}|\mathcal{N}_j(\chi, k)|^{2}\Big )|{\mathcal
 Q}_{v+1}(\chi, k)|^2   \ll \phis(q)(\log q)^{ k^2 }.
\end{align*}
\end{proposition}

   Note that Proposition \ref{Prop6} is contained in \cite[Proposition 3.5]{Gao2024}.  It therefore remains to prove Proposition \ref{Prop5} in the sequel.

\subsection{Proof of Proposition \ref{Prop5}}
\label{sec 5}

   We write that
\begin{align}
\label{Pexpression}
  {\mathcal P}_{v+1}(\chi)^{r_k\ell_{v+1}} =&  \sum_{ \substack{ n_{v+1}}} \frac{1}{\sqrt{n_{v+1}}}\frac{(r_k\ell_{v+1})!
  }{w(n_{v+1})}\chi(n_{v+1})p_{v+1}(n_{v+1}),
\end{align}
  where $w(n)$ is defined to be the multiplicative function satisfying $w(p^{\alpha}) = \alpha!$ for prime powers $p^{\alpha}$ and $p_{v+1}(n) \in \{ 0,1 \}$ with $p_{v+1}(n)=1$ if and only if $n$ is composed of exactly $r_k\ell_{v+1}$ primes (counted with multiplicity), all from the interval $P_{v+1}$. \newline
  
     We then proceed as in the proof of \cite[Proposition 3.4]{Gao2024} to see that we may write for simplicity, 
\begin{align}
\label{ProdNQ}
\begin{split} \Big (\prod^v_{j=1}|\mathcal{N}_j(\chi, k-2)|^2 \Big )|{\mathcal
 Q}_{v+1}(\chi, k)|^{2} = \Big( \frac{12  }{\ell_{v+1}}\Big)^{2r_k\ell_{v+1}}((r_k\ell_{v+1})!)^2 \sum_{a,b \leq q^{2r_k/10^{M}}} \frac{u_a u_b}{\sqrt{ab}}\chi(a)\overline{\chi}(b).
 \end{split}
\end{align}
Here the coefficients $u_a$ and $u_b$ emerge from the direct expansion of the left-hand side of \eqref{ProdNQ} using \eqref{defPQ}, \eqref{defN} and \eqref{Pexpression}.  An important property to keep in mind is that $u_a, u_b \leq 1$ for all $a, b$. Here we note that by \cite[(5.3)]{Gao2024}, 
\begin{align}
\label{factorbound}
 & \Big( \frac{12 }{\ell_{v+1}} \Big)^{2r_k\ell_{v+1}}((r_k \ell_{v+1})!)^2
\leq (r_k\ell_{v+1})^2 \Big( \frac{12 r_k }{e } \Big)^{2r_k\ell_{v+1}}.
\end{align}

Note further that 
\begin{align*}
&\sumstar_{\substack{ \chi \shortmod q }}|L(\tfrac{1}{2},\chi)|^4 \sum^{R}_{v=0}\Big (\prod^v_{j=1}|\mathcal{N}_j(\chi, k-2)|^{2}\Big ) |{\mathcal
 Q}_{v+1}(\chi, k)|^2  =\sum^{R}_{v=0}S_{\pm},
\end{align*}  
  where
\begin{align*}
     S_{\pm} = & \sumstar_{\substack{ \chi \shortmod q \\ \chi(-1)=\pm 1}}|L(\tfrac{1}{2},\chi)|^4 \Big (\prod^v_{j=1}|\mathcal{N}_j(\chi, k-2)|^{2}\Big ) |{\mathcal
 Q}_{v+1}(\chi, k)|^2. 
\end{align*}  
  As the treatments are similar, we will focus on $S_+$ in what follows.  From \eqref{ProdNQ},
\begin{align} \label{LNsquaresum}
S_+ =  \Big( \frac{12  }{\ell_{v+1}}\Big)^{2r_k\ell_{v+1}}((r_k\ell_{v+1})!)^2 \sum_{a,b \leq q^{2r_k/10^{M}}} \frac{u_a u_b}{\sqrt{ab}} 
\sumstar_{\substack{ \chi \shortmod q \\ \chi(-1)=1}}  |L( \tfrac12, \chi)|^4 \chi(a) \overline{\chi}(b).
\end{align}

We take $M$ large enough so that $a, b \leq q^{2r_k/10^{M}}<q^{\vartheta}<q$.  Next note that $\chi(a) \overline{\chi}(b)=\chi(a/(a,b)) \overline{\chi}(b/(a,b))$, where $(a,b)$ denotes the greatest common divisor of $a$ and $b$.  Since $(a/(a,b), b/(a,b))=1$, we are in the position to apply Proposition  \ref{twisted_fourth_moment_theorem on the central point} to evaluate the last sum in \eqref{LNsquaresum}. As shown in the proof of \cite[Proposition 3.4]{Gao2024}, the contribution from the $O$-term in \eqref{twistedfourthmoment} is negligible because of the power saving, upon taking $M$ large enough. Thus, we may concentrate on the main term \eqref{twistedfourthmoment}.  Using \eqref{factorbound}, 
\begin{align}
\label{LNsquaresum1}
\begin{split}
S_+ \ll  \phis(q) (r_k\ell_{v+1})^2 \Big( \frac{12 r_k }{e } \Big)^{2r_k\ell_{v+1}} & \sum_{a,b \leq q^{2r_k/10^{M}}} \frac{(a,b)\tau_{0}(\frac a{(a,b)})\tau_{ 0}(\frac b{(a,b)})}{ab} u_a u_b \\
& \times \Big (\sum_{\substack{0 \leq j_0+j_1+\ldots+j_4 \leq 4 \\ j_0 \geq 0, j_1 \geq 0, \ldots, j_4 \geq 0}}(\log q)^{j_0}
\sum_{\substack{ 1 \leq i \leq 4 \\ p^{\nu_i}_i \| \frac {ab}{(a,b)^2}, \nu_i \geq 1 \\ p_i \text{distinct} }}  C_{j_0, \ldots, j_4,i}(p_i)(\log p^{\nu_i}_i)^{j_i}\Big ).
\end{split}
\end{align}

As the estimations of the complementary are similar, it suffices to consider the following subsum of the right-hand side of \eqref{LNsquaresum1} given by
\begin{align}
\label{sumoverlog}
\begin{split}
\phis(q) & (\log q)^3(r_k\ell_{v+1})^2 \Big( \frac{12 r_k }{e } \Big)^{2r_k\ell_{v+1}} \sum_{a,b }  \sum_{p^{l_1} \| a} \frac{(a,b)\tau_{0}(\frac a{(a,b)})\tau_{ 0}(\frac b{(a,b)})}{ab} u_a u_b \log p^{l_1} \\
=& \phis(q)(\log q)^3(r_k\ell_{v+1})^2 \Big( \frac{12 r_k }{e } \Big)^{2r_k\ell_{v+1}} \sum_{p \in \bigcup^{v+1}_{j=1} P_j}\sum_{l_1 \geq 1, l_2 \geq 0}\frac {l_1 \log p\tau_{0}(p^{l_1-\min (l_1, l_2)})\tau_{0}(p^{l_2-\min (l_1, l_2)})}{p^{l_1+l_2-\min (l_1, l_2)}}\frac {(k-2)^{l_1+l_2}}{l_1!l_2!}\\
& \hspace*{2cm} \times \sum_{\substack{ a,b \\ (ab, p)=1} } \frac{(a,b) \tau_{0}(\frac a{(a,b)})\tau_{ 0}(\frac b{(a,b)}) u_{p^{l_1}a} u_{p^{l_2}b}}{ab} .
\end{split}
\end{align}
   We now treat the last sum above for fixed $p=p_1, l_1, l_2$. Without loss of generality, we may assume that $p=p_1 \in P_1$.  Let $\Omega(n)$ denote the number of distinct prime powers dividing $n$.  We define for $(n_1n'_1, p_1)=1$,
\begin{equation*}
 v_{n_1} = \frac{1}{n_1} \frac{(k-2)^{\Omega(n_1)}}{w(n_1)}  b_1(n_1p_1^{l_1}), \quad v_{n'_1} = \frac{1}{n'_1} \frac{(k-2)^{\Omega(n'_1)}}{w(n'_1)}  b_1(n'_1p_1^{l_2}).
\end{equation*}
  Also, for $2 \leq j \leq v$,
\begin{equation*}
 v_{n_j} = \frac{1}{n_j} \frac{(k-2)^{\Omega(n_j)}}{w(n_j)}  b_j(n_j) \quad \mbox{and} \quad v_{n'_j} = \frac{1}{n'_j} \frac{(k-2)^{\Omega(n'_j)}}{w(n'_j)}  b_j(n'_j)
\end{equation*}
and
\begin{equation*}
 v_{n_{v+1}} = \frac{1}{n_{v+1}} \frac{1}{w(n_{v+1})}  p_{v+1}(n_{v+1}), \quad  v_{n'_{v+1}} = \frac{1}{n'_{v+1}} \frac{1}{w(n'_{v+1})}  p_{v+1}(n'_{v+1}).
\end{equation*}
Thus, using these notations,
\begin{align}
\label{doublesum}
 \sum_{\substack{ a,b \\ (ab, p_1)=1} } \frac{(a,b)\tau_{0}(\frac a{(a,b)})\tau_{ 0}(\frac b{(a,b)}) u_{p^{l_1}a} u_{p^{l_2}b}}{ab}=\prod^{v+1}_{j=1}\Big ( \sum_{\substack{n_j, n'_j \\ (n_1n'_1, p_1)=1}}(n_j, n'_j)\tau_{0}(\frac {n_j}{(n_j, n'_j)})\tau_{ 0}(\frac {n'_j}{(n_j, n'_j)})v_{n_j}v_{n'_j} \Big ).
\end{align}

   As in the proof of \cite[Proposition 3.4]{Gao2024}, when $\max (l_1, l_2) \leq \ell_1/2$, we remove the restriction of $b_1(n_1)$ on $\Omega_1(n_1)$ and $b_1(n'_1)$ on $\Omega_1(n'_1)$ to deduce that the last sum over $n_1, n'_1$ in \eqref{doublesum} becomes
\begin{align*}
  \sum_{\substack{n_1, n'_1\\ (n_1n'_1, p_1)=1}}(n_1, n'_1) & \tau_{0}\Big(\frac {n_1}{(n_1, n'_1)}\Big)\tau_{ 0} \Big(\frac {n'_1}{(n_1, n'_1)}\Big)\frac{1}{n_1} \frac{(k-2)^{\Omega(n_1)}}{w(n_1)} \frac{1}{n'_1} \frac{(k-2)^{\Omega(n'_1)}}{w(n'_1)} \\
   & =
  \prod_{\substack{p \in P_1 \\ p \neq p_1}}\Big ( 1+\frac{4(k-2)+(k-2)^2}{p} +O\Big(\frac 1{p^2}\Big)  \Big ) \ll \exp \Big(\sum_{p \in P_1}\frac{k^2-4}{p}+O\Big(\sum_{p \in P_1}\frac 1{p^2}\Big)\Big),
\end{align*}
  where the last bound above follows by noting that $\tau_0$ is a multiplicative function satisfying, 
\begin{align*}
  \tau_0(p^{\nu})=1+\nu+O\Big(\frac 1p\Big), \quad \mbox{for any integer} \; \nu \geq 0.
\end{align*}  

  Meanwhile, we observe that $2^{\Omega(n)-\ell_1/2}\ge 1$ if $\Omega(n)+\max (l_1, l_2) \geq \ell_1$.  Thus, applying Rankin's trick reveals that the error incurred in the afore-mentioned removal does not exceed
\begin{equation} \label{removalerror}
\begin{split}
  2^{-\ell_1/2} \sum_{n_1, n'_1}\frac {(n_1, n'_1)}{n_1n'_1}\tau_{0} & \Big(\frac {n_1}{(n_1, n'_1)}\Big)\tau_{ 0}\Big(\frac {n'_1}{(n_1, n'_1)}\Big) \frac{(2-k)^{\Omega(n_1)}}{w(n_1)} \frac{(2-k)^{\Omega(n'_1)}2^{\Omega(n'_1)}}{w(n'_1)} \\
 =& 2^{-\ell_1/2} \prod_{p \in P_1}\Big ( 1+\frac{6(2-k)+2(2-k)^2}{p} +O\Big(\frac 1{p^2}\Big)  \Big ).
 \end{split}
\end{equation}
   As in \cite[(4.5)]{Gao2024}, upon taking $N$ large enough,
\begin{align}
\label{boundsforsumoverp}
   \sum_{p \in P_j}\frac 1{p} \leq \frac 1N \ell_j.
\end{align}   
From the above, it follows that when $\max (l_1, l_2) \leq \ell_1/2$, \eqref{removalerror} is
\begin{align*}
 \ll &  2^{-\ell_1/4} \exp \Big(\sum_{p \in P_1}\frac{k^2-4}{p}+O\Big(\sum_{p \in P_1}\frac 1{p^2} \Big)\Big).
\end{align*}

   On the other hand, \eqref{boundsforsumoverp} again leads to, when $\max (l_1, l_2) > \ell_1/2$, 
\begin{align*}
  \sum_{\substack{n_1, n'_1 \\ (n_1n'_1, p_1)=1}} & (n_1, n'_1)\tau_{0}\Big(\frac {n_1}{(n_1, n'_1)}\Big)\tau_{ 0}\Big(\frac {n'_1}{(n_1, n'_1)}\Big) v_{n_1}v_{n'_1} \\
 \ll &  \sum_{\substack{n_1, n'_1 \\ (n_1n'_1, p_1)=1}}\frac{(n_1, n'_1)}{n_1n'_1}\tau_{0}\Big(\frac {n_1}{(n_1, n'_1)}\Big)\tau_{ 0}\Big(\frac {n'_1}{(n_1, n'_1)}\Big) 
 \frac{(2-k)^{\Omega(n_1)}}{w(n_1)}  \frac{(2-k)^{\Omega(n'_1)}}{w(n'_1)} \\
 \ll & 2^{\ell_1/6}  \exp \Big(\sum_{p \in P_1}\frac{k^2-3}{p}+O\Big(\sum_{p \in P_1}\frac
 1{p^2}\Big)\Big).
\end{align*}
   We then derive that in this case (noting that $(k-2)^{l_1+l_2} \ll 1$ for $1 < k<2$)
\begin{align*}
 &  \frac {l_1 \log p_1\tau_{0}(p^{l_1-\min (l_1, l_2)})\tau_{0}(p^{l_2-\min (l_1, l_2)})}{p^{l_1+l_2-\min (l_1, l_2)}}\frac {(k-2)^{l_1+l_2}}{l_1!l_2!}\sum_{\substack{n_1, n'_1 \\ (n_1n'_1, p_1)=1}}(n_1, n'_1)\tau_{0}\Big(\frac {n_1}{(n_1, n'_1)}\Big)\tau_{ 0}\Big(\frac {n'_1}{(n_1, n'_1)}\Big) v_{n_1}v_{n'_1} \\
& \hspace*{1cm} \ll  \frac {l_1 \log p_1 \tau_{0}(p^{l_1-\min (l_1, l_2)})\tau_{0}(p^{l_2-\min (l_1, l_2)})}{p^{\max (l_1, l_2)}_1} \frac {1}{l_1!l_2!} 2^{\ell_1/6}  \exp \Big(\sum_{p \in P_1}\frac{k^2-4}{p}+O\Big(\sum_{p \in P_1}\frac 1{p^2}\Big)\Big) \\
 & \hspace*{1cm} \ll \frac {l_1 (|l_2-l_1|+1) \log p_1}{p^{l_1/2+\max (l_1, l_2)/2}_1} \frac {1}{l_1!l_2!} 2^{\ell_1/6} \exp \Big(\sum_{p \in P_1}\frac{k^2-4}{p}+O\Big(\sum_{p \in
 P_1}\frac 1{p^2}\Big)\Big) \\
 & \hspace*{1cm} \ll \frac {l_1 \log p_1}{p^{l_1/2}_1}\frac {1}{l_1!l_2!} 2^{-\ell_1/20}   \exp \Big(\sum_{p \in  P_1}\frac{k^2-4}{p_1}+O\Big(\sum_{p \in  P_1}\frac 1{p^2}\Big)\Big).
\end{align*}

   Similar estimations apply to the sums over $n_j, n'_j$ for $2 \leq j \leq v$ in \eqref{doublesum}.  We then apply Rankin's trick again to deal with the sum over $n_{v+1}$, $n'_{v+1}$.  It is
\begin{align*}
 \ll &  (12r_k)^{-2r_k\ell_{v+1}} \sum_{\substack{n_{v+1}, n'_{v+1} \\ p |
n_{v+1}n'_{v+1} \implies  p\in P_{v+1} } }\frac {(n_{v+1}, n'_{v+1})}{n_{v+1}n'_{v+1}}\tau_{0}\Big(\frac {n_{v+1}}{(n_{v+1}, n'_{v+1})}\Big)\tau_{ 0}\Big(\frac {n'_{v+1}}{(n_{v+1}, n'_{v+1})}\Big) \frac{(12r_k)^{\Omega(n_{v+1})}}{w(n_{v+1})} \frac{(12r_k)^{\Omega(n'_{v+1})}}{w(n'_{v+1})}.
\end{align*}
   We take $N$ large enough to see that the above implies that
\begin{align*}
\begin{split}
  (r_k\ell_{v+1})^2 \Big( \frac{12 r_k }{e } \Big)^{2r_k\ell_{v+1}} & \sum_{\substack{n_{v+1}, n'_{v+1} \\ (n_{v+1}n'_{v+1}, p_1)=1}}(n_{v+1}, n'_{v+1})\tau_{0}\Big(\frac {n_{v+1}}{(n_{v+1}, n'_{v+1})}\Big)\tau_{ 0}\Big(\frac {n'_{v+1}}{(n_{v+1}, n'_{v+1})}\Big)v_{n_{v+1}}v_{n'_{v+1}} \\
\ll &  e^{-\ell_{v+1}}\exp \Big(\sum_{p \in P_{v+1}}\frac{k^2-4}{p}+O\Big(\sum_{p \in P_{v+1}}\frac 1{p^2}\Big)\Big).
\end{split}
\end{align*}

  We thus deduce from the above discussions that 
\begin{align}
\label{sumpbound}
\begin{split}
 (r_k\ell_{v+1})^2 & \Big( \frac{12 r_k }{e } \Big)^{2r_k\ell_{v+1}} \frac {l_1 \log p_1\tau_{0}(p^{l_1-\min (l_1, l_2)}_1)\tau_{0}(p^{l_2-\min (l_1, l_2)}_1)}{p^{l_1+l_2-\min (l_1, l_2)}_1}\frac {(k-2)^{l_1+l_2}}{l_1!l_2!}\\
& \hspace*{2cm} \times \sum_{\substack{ a,b \\ (ab, p_1)=1} } \frac{(a,b) \tau_{0}(\frac a{(a,b)})\tau_{ 0}(\frac b{(a,b)}) u_{p^{l_1}_1a} u_{p^{l_2}_1b}}{ab} \\
 \ll & e^{-\ell_{v+1}}\prod^v_{j=1}\big(1+O(2^{-\ell_j/20})\big)\exp \Big(\sum_{p \in \bigcup^{v+1}_{j=1} P_j}\frac{k^2-4}{p}+O\Big(\sum_{p \in \bigcup^{v+1}_{j=1} P_j}\frac 1{p^2}\Big)\Big) \Big ( \frac{\log p_1}{p_1} + O\Big(\frac
 {\log p_1}{p^2_1}\Big) \Big ).
\end{split}
\end{align}

  By \eqref{boundsforsumoverp} and the observation $\ell_j > \ell^2_{j+1}>2\ell_{j+1}$, we obtain that
\begin{align}
\label{sumpbound1}
 \sum_{p \in \bigcup^R_{j=v+2} P_j}\frac{|k^2-4|}{p} \leq \sum^{R}_{j=v+2}\sum_{p \in P_j}\frac 1{p} \leq \frac 1N \sum^{R}_{j=v+2} \ell_j \leq \frac {\ell_{v+2}}{N} \sum^{\infty}_{j=0} \frac 1{2^j} \leq \frac {\ell_{v+1}}{N}.
\end{align}

Now \eqref{sumpbound1} renders that the last expression in \eqref{sumpbound} is
\begin{align}
\label{sumpbound2}
\begin{split}
 \ll &  e^{-\ell_{v+1}/2}\exp \Big(\sum_{p \in \bigcup^{R}_{j=1} P_j}\frac{k^2-4}{p}+O\Big(\sum_{p \in \bigcup^{R}_{j=1} P_j}\frac 1{p^2}\Big)\Big) \Big ( \frac{\log p_1}{p_1} + O\Big(\frac
 {\log p_1}{p^2_1}\Big) \Big ).
\end{split}
\end{align}

From \cite[Theorem 2.7]{MVa1}, we have for some $x \geq 2$ and some constant $b$,
\begin{align*}
\begin{split}
\sum_{p\le x} \frac{1}{p} = \log \log x + b+ O\Big(\frac{1}{\log x}\Big) \quad \mbox{and} \quad
\sum_{p\le x} \frac {\log p}{p} =  \log x + O(1).
\end{split}
\end{align*}

 We then conclude from this, \eqref{LNsquaresum1}, \eqref{sumoverlog} and \eqref{sumpbound2} that
\begin{align*}
\begin{split}
S_+ \ll &    \phis(q)(\log q)^3 e^{-\ell_{v+1}/2}\exp \Big(\sum_{p \in \bigcup^{R}_{j=1} P_j}\frac{k^2-4}{p}+O\Big(\sum_{p \in \bigcup^{R}_{j=1} P_j}\frac 1{p^2}\Big)\Big)  \sum_{p \in \bigcup^{v+1}_{j=1}P_j} \Big ( \frac{\log p}{p} + O\Big(\frac
 {\log p}{p^2}\Big) \Big ) \\
\ll & \phis(q) e^{-\ell_{v+1}/2}(\log q)^{k^2}.
\end{split}
\end{align*}
  As the sum over $e^{-\ell_{v+1}/2}$ converges,
\begin{align*}
\begin{split}
 \sum^{R}_{v=0}S_+ \ll \phis(q) e^{-\ell_{v+1}/2}(\log q)^{k^2}.
\end{split}
\end{align*}  
A similar estimate holds with $S_+$ replaced by $S_{-}$. This completes the proof of the proposition.

\vspace*{.5cm}

\noindent{\bf Acknowledgments.} P. G. is supported in part by NSFC grant 12471003 and L. Z. by the FRG Grant PS71536 at the University of New South Wales.
The authors are grateful to the referee for a careful reading of the paper and many helpful comments.

\bibliography{biblio}

\begin{bibdiv}
\begin{biblist}

\bib{AC25}{article}{
      author={Arguin, L.-P.},
      author={Creighton, N.},
       title={Upper bounds on large deviations of {D}irichlet {$L$}-functions
  in the $q$-aspect},
        date={2025},
     journal={J. Number Theory},
      volume={273},
       pages={96–158},
}

\bib{BFKMM}{article}{
      author={Blomer, V.},
      author={Fouvry, \'{E}.},
      author={Kowalski, E.},
      author={Michel, P.},
      author={Mili\'{c}evi\'{c}, D.},
       title={On moments of twisted {$L$}-functions},
        date={2017},
        ISSN={0002-9327},
     journal={Amer. J. Math.},
      volume={139},
      number={3},
       pages={707\ndash 768},
}

\bib{BFKMM1}{article}{
      author={Blomer, V.},
      author={Fouvry, \'{E}.},
      author={Kowalski, E.},
      author={Michel, P.},
      author={Mili\'{c}evi\'{c}, D.},
       title={Some applications of smooth bilinear forms with {K}loosterman
  sums},
        date={2017},
     journal={Tr. Mat. Inst. Steklova},
      volume={296},
       pages={Analiticheskaya i Kombinatornaya Teoriya Chisel, 24\ndash 35},
        note={English version published in Proc. Steklov Inst. Math.
  {{\bf{2}}96} (2017), no. 1, 18--29},
}

\bib{BFKMMS}{article}{
      author={Blomer, V.},
      author={Fouvry, \'E.},
      author={Kowalski, E.},
      author={Michel, P.},
      author={Mili\'cevi\'c, D.},
      author={Sawin, W.},
       title={The second moment theory of families of {$L$}-functions---the
  case of twisted {H}ecke {$L$}-functions},
        date={2023},
     journal={Mem. Amer. Math. Soc.},
      volume={282},
      number={1394},
       pages={v+148},
}

\bib{BM15}{article}{
      author={Blomer, V.},
      author={Mili\'{c}evi\'{c}, D.},
       title={The second moment of twisted modular {$L$}-functions},
        date={2015},
     journal={Geom. Funct. Anal.},
      volume={25},
      number={2},
       pages={453\ndash 516},
}

\bib{BPRZ}{article}{
      author={Bui, H.~M.},
      author={Pratt, K.},
      author={Robles, N.},
      author={Zaharescu, A.},
       title={Breaking the {$\frac12$}-barrier for the twisted second moment of
  {D}irichlet {$L$}-functions},
        date={2020},
     journal={Adv. Math.},
      volume={370},
       pages={107175, 40pp},
}

\bib{C&L}{article}{
      author={Chandee, V.},
      author={Li, X.},
       title={Lower bounds for small fractional moments of {D}irichlet
  {$L$}-functions},
        date={2013},
     journal={Int. Math. Res. Not. IMRN},
      number={19},
       pages={4349\ndash 4381},
}

\bib{CFKRS}{article}{
      author={Conrey, J.~B.},
      author={Farmer, D.~W.},
      author={Keating, J.~P.},
      author={Rubinstein, M.~O.},
      author={Snaith, N.~C.},
       title={Integral moments of {$L$}-functions},
        date={2005},
     journal={Proc. London Math. Soc. (3)},
      volume={91},
      number={1},
       pages={33\ndash 104},
}

\bib{Gao2024}{article}{
      author={Gao, P.},
       title={Bounds for moments of {D}irichlet {$L$}-functions to a fixed
  modulus},
        date={2024},
     journal={Math. Z.},
      volume={307},
      number={4},
       pages={Paper No. 70, 18pp},
}

\bib{Harper}{article}{
      author={Harper, A.~J.},
       title={Sharp conditional bounds for moments of the {R}iemann zeta
  function},
        date={Preprint},
        note={arXiv:1305.4618},
}

\bib{H&Sound}{article}{
      author={Heap, W.},
      author={Soundararajan, K.},
       title={Lower bounds for moments of zeta and {$L$}-functions revisited},
        date={2022},
     journal={Mathematika},
      volume={68},
      number={1},
       pages={1\ndash 14},
}

\bib{HB81}{article}{
      author={Heath-Brown, D.~R.},
       title={The fourth power mean of {D}irichlet's {$L$}-functions},
        date={1981},
     journal={Analysis},
      volume={1},
      number={1},
       pages={25\ndash 32},
}

\bib{HB2010}{article}{
      author={Heath-Brown, D.~R.},
       title={Fractional moments of {D}irichlet {$L$}-functions},
        date={2010},
     journal={Acta Arith.},
      volume={145},
      number={4},
       pages={397\ndash 409},
}

\bib{Hough2016}{article}{
      author={Hough, B.},
       title={The angle of large values of {$L$}-functions},
        date={2016},
     journal={J. Number Theory},
      volume={167},
       pages={353\ndash 393},
}

\bib{Liu24}{book}{
      author={Liu, D.},
       title={Zeros and moments of {$L$}-functions and applications},
        date={2024},
        note={Thesis (Ph.D.)--University of Illinois Urbana-Champaign, 95pp},
}

\bib{MVa1}{book}{
      author={Montgomery, H.~L.},
      author={Vaughan, R.~C.},
       title={{M}ultiplicative number theory. {I}. {C}lassical theory},
      series={Cambridge Studies in Advanced Mathematics},
   publisher={Cambridge University Press},
     address={Cambridge},
        date={2007},
      volume={97},
}

\bib{Radziwill&Sound1}{article}{
      author={Radziwi{\l\l}, M.},
      author={Soundararajan, K.},
       title={Continuous lower bounds for moments of zeta and {$L$}-functions},
        date={2013},
     journal={Mathematika},
      volume={59},
      number={1},
       pages={119\ndash 128},
}

\bib{Radziwill&Sound}{article}{
      author={Radziwi{\l \l}, M.},
      author={Soundararajan, K.},
       title={Moments and distribution of central {$L$}-values of quadratic
  twists of elliptic curves},
        date={2015},
     journal={Invent. Math.},
      volume={202},
      number={3},
       pages={1029\ndash 1068},
}

\bib{R&Sound1}{article}{
      author={Rudnick, Z.},
      author={Soundararajan, K.},
       title={Lower bounds for moments of {$L$}-functions: symplectic and
  orthogonal examples},
     journal={in: Multiple {D}irichlet series, automorphic forms, and analytic
  number theory, 293--303, Proc. Sympos. Pure Math.},
      volume={75},
       pages={Amer. Math. Soc., Providence, RI, 2006.},
}

\bib{R&Sound}{article}{
      author={Rudnick, Z.},
      author={Soundararajan, K.},
       title={Lower bounds for moments of {$L$}-functions},
        date={2005},
     journal={Proc. Natl. Acad. Sci. USA},
      volume={102},
      number={19},
       pages={6837\ndash 6838},
}

\bib{Selberg46}{article}{
      author={Selberg, A.},
       title={Contributions to the theory of {D}irichlet's {$L$}-functions},
        date={1946},
     journal={Skr. Norske Vid.-Akad. Oslo I},
      volume={1946},
      number={3},
       pages={62pp},
}

\bib{Sound2007}{article}{
      author={Soundararajan, K.},
       title={The fourth moment of {D}irichlet {$L$}-functions},
     journal={in: Analytic number theory, 239–246, Clay Math. Proc.,},
      volume={7},
       pages={Amer. Math. Soc., Providence, RI, 2007.},
}

\bib{Sound01}{article}{
      author={Soundararajan, K.},
       title={Moments of the {R}iemann zeta function},
        date={2009},
     journal={Ann. of Math. (2)},
      volume={170},
      number={2},
       pages={981\ndash 993},
}

\bib{Wu2020}{article}{
      author={Wu, X.},
       title={The fourth moment of {D}irichlet {$L$}-functions at the central
  value},
        date={2023},
     journal={Math. Ann.},
      volume={387},
      number={3-4},
       pages={1199\ndash 1248},
}

\bib{Young2011}{article}{
      author={Young, M.~P.},
       title={The fourth moment of {D}irichlet {$L$}-functions},
        date={2011},
     journal={Ann. of Math. (2)},
      volume={173},
      number={1},
       pages={1\ndash 50},
}

\bib{Z2019}{article}{
      author={Zacharias, R.},
       title={Mollification of the fourth moment of {D}irichlet
  {$L$}-functions},
        date={2019},
     journal={Acta Arith.},
      volume={191},
      number={3},
       pages={201\ndash 257},
}

\end{biblist}
\end{bibdiv}
\bibliographystyle{amsxport}

\vspace*{.5cm}

\end{document}